\documentclass[11pt,a4paper]{amsart}
\usepackage[utf8]{inputenc}
\usepackage{faktor}
\usepackage[T1]{fontenc}
\usepackage{lmodern}
\usepackage[english]{babel}
\usepackage{amsmath}
\usepackage{xfrac}
\usepackage{esint}
\usepackage{stmaryrd,mathrsfs,bm,amsthm,mathtools,yfonts,amssymb,color,braket,booktabs,graphicx,graphics,amsfonts}
\usepackage{latexsym,microtype,indentfirst,hyperref}
\usepackage{xcolor}
\usepackage{courier}
\usepackage[colorinlistoftodos]{todonotes}

\title[Stability for anisotropic nearly umbilical hypersurfaces]{Quantitative stability for anisotropic nearly umbilical hypersurfaces }
\author[A. De Rosa]{Antonio De Rosa}
\address{A.D.R.: Institut f\"ur Mathematik, Universit\"at Z\"urich, Winterthurerstrasse 190, CH-8057 Z\"urich, Switzerland}
\email{antonio.derosa@math.uzh.ch}

\author[S. Gioffrè]{Stefano Gioffrè}
\address{S.G.: Institut f\"ur Mathematik, Universit\"at Z\"urich, Winterthurerstrasse 190, CH-8057 Z\"urich, Switzerland}
\email{stefano.gioffre@math.uzh.ch}

\newtheorem{teo}{Theorem}[section]
\newtheorem{defi}{Definition}[section]
\newtheorem{prop}[teo]{Proposition}

\newtheorem{lemma}[teo]{Lemma}

\numberwithin{equation}{section}

\newcommand{\N}{\mathbb{N}}

\newcommand{\R}{\mathbb{R}}

\newcommand{\erre}{\mathbb{R}}
\newcommand{\esse}{\mathbb{S}}

\newcommand{\W}{\mathcal{W}}
\newcommand{\daA}[2]{\colon #1 \longrightarrow #2}

\newcommand{\restr}[2]{\left. #1 \right|_{#2}}

\newcommand{\Sdot}{\mathring{S}}
\newcommand{\PF}{\mathcal{F}}
\newcommand{\Per}{\mathcal{P}}

\newcommand{\gu}{\mathfrak{u}}

\renewcommand{\epsilon}{\varepsilon}
\renewcommand{\phi}{\varphi}
\renewcommand{\theta}{\vartheta}

\DeclarePairedDelimiter{\abs}{|}{|}
\DeclarePairedDelimiter{\coup}{(}{)}				

\DeclarePairedDelimiter{\norm}{\lVert}{\rVert}

\DeclareMathOperator{\Id}{Id}
\DeclareMathOperator{\id}{Id}
\DeclareMathOperator{\tr}{tr}

\DeclareMathOperator{\divv}{div}

\DeclareMathOperator{\hd}{HD}

\DeclareMathOperator{\Lip}{Lip}
\DeclareMathOperator{\Span}{Span}

\date{}
\allowdisplaybreaks
\begin{document}
\begin{abstract}
We prove a qualitative and a quantitative stability of the following rigidity theorem: an anisotropic totally umbilical closed hypersurface is the Wulff shape. Consider $n \ge 2$, $p\in (1, \, +\infty)$ and $\Sigma$ an $n$-dimensional, closed hypersurface in $\R^{n+1}$,  boundary of a convex, open set. We show that if the $L^p$ norm of the trace-free part of the anisotropic second fundamental form is small, then  $\Sigma$ must be $W^{2, \, p}$-close to the Wulff shape, with a quantitative estimate.
\end{abstract}
\maketitle

\section{Introduction}
Let $\Sigma$ be a smooth, connected surface in $\R^{3}$. 
A point on $\Sigma$ is called umbilical if the two principal curvatures in this point coincide. The classical umbilical theorem (see for istance \cite[Chapter 3.2, Proposition 4]{doCarmo}) states that, if all the points
of $\Sigma$ are umbilical, then $\Sigma$ is either a subset
of a plane or of a sphere. In particular, if $\Sigma$ is closed, it can just coincide with a sphere.
This rigidity theorem was extended to $\R^{n+1}$ in \cite[Lemma 1, p. 8]{Spivak99}. These results are relevant since a local condition, such as the pointwise umbilicality, gives a strong global information.

The umbilical theorem has been strenghten with qualitative stability results in \cite{DLM,DLC0, Daniel}, which have produced important applications in the foliations of asymptotically flat three–manifolds by surfaces of prescribed mean curvature (see \cite{Met,LMS,LM},
where the aforementioned results are crucial for ensuring that the leaves are close to
spheres). A quantitative stability has been later obtained in \cite{Gioffre2016}.

The rigidity umbilical theorem has been proved to be true even in the anisotropic setting in \cite{HeLi}: namely it is shown that the only closed hypersurface with diagonal anisotropic second fundamental form is the Wulff shape. Instead, the stability of this result, even at a qualitative level, has not been addressed in the litearture.

The aim of this paper is to obtain the anisotropic counterparts of the qualitative and quantitative stability proved in \cite{DLC0, Daniel,Gioffre2016} for the isotropic setting. Namely, we consider $p\in (1, \, +\infty)$ and $\Sigma$ an $n$-dimensional, closed hypersurface in $\R^{n+1}$,  boundary of a convex, open set. We prove that the $W^{2, \, p}$-closedness of $\Sigma$ to the Wulff shape is controlled by the $L^p$ norm of the trace-free part of the anisotropic second fundamental form.

Throughout the paper, we will say that $\Sigma$ is convex if it is boundary of a convex bounded open set. The convexity assumption on the surface $\Sigma$ is necessary in order to avoid bubbling phenomena. Indeed, in the recent paper \cite{DMMN}, it is proven that if $\Sigma$ is a closed hypersurface (not necessarily convex) with anisotropic mean curvature $L^2$-close to a costant, then $\Sigma$ is $L^1$-close to a finite union of Wulff shapes.

In order to state our main result, we need some notation. We consider a smooth anisotropic function defined on the $n$-sphere:
$$F \daA{\esse^n}{(0, \, \infty)}.$$
For every  smooth hypersurface $\Sigma$ in $\erre^{n+1}$, we define its \textit{anisotropic surface energy} as
\begin{equation}\label{AnisoPerim}
\PF (\Sigma):= \int_\Sigma F(\nu_\Sigma) \, dV.
\end{equation}
 In particular, we will denote the isotropic surface energy as
$\Per (\Sigma)$, to be the energy $\PF (\Sigma)$ with $F(\nu)=|\nu|$.

For every Lagrangian $F$ and every $m>0$, it is natural to ask whether the problem 
\begin{equation}\label{AnisoProblem}
\inf \set{ \PF(\Sigma) : \Sigma = \partial U, \, |U|=m}
\end{equation}
admits a solution. The answer is positive, provided that $F$ satisfies an ellipticity condition. 

If we denote by $A^F$ the quantity
\begin{equation}\label{AF}
\restr{A^F}{x}[z] := \restr{D^2 F}{x}[z] + F(x)z \quad \mbox{ for every } x \in \esse^n, \, z \in T_x \esse^n,
\end{equation}
we say that $F$ is an \textit{elliptic integrand} if the symmetric matrix $A^F$ is positive definite at every $x \in \esse^n$. In this case, problem \eqref{AnisoProblem} is solvable, and its solution is a dilation of a closed, convex hypersurface $\W$ called \textit{Wulff shape}. In the context of differential geometry, the Wulff shape shares lots of similarities with the round sphere. For instance, it can be seen as the ``round sphere'' for an anisotropic norm on $\R^{n+1}$, namely 
\begin{equation}\label{gauge}
\W = \set{F^* = 1},
\end{equation}
where $F^*$ is the gauge function $F^* : \R^{n+1} \mapsto [0,+\infty)$ defined as
$$F^*(x) := \sup_{\nu \in \erre^{n+1}} \left\{\langle x, \nu \rangle : |\nu|F\left(\frac{\nu}{|\nu|}\right) \leq 1\right\}.$$
A useful property of the differential of the gauge function, that we will implement later, is (see for instance \cite[p. 8]{Robin}): 
\begin{equation}\label{Robin}
\restr{dF^*}{z}[c] = \langle \nu_\W(z), \, c\rangle , \qquad \forall z \in \W,
\end{equation}
where $\nu_\W$ is the normal vector field associated to $\W$.

We have already recalled that an anisotropic version of the umbilical theorem exists, and $\W$ can be characterized as the only hypersurface with diagonal anisotropic second fundamental form; namely $\W$ is the only anisotropic totally umbilical hypersurface, as stated in the following theorem proved in \cite{HeLi}:
\begin{teo}\label{AnisoSecondFFTheorem}
Let $n \ge 2$ be given, and let $\Sigma$ be a closed, oriented hypersurface with outer normal $\nu_{\Sigma}$. Denote the anisotropic second fundamental form as
\begin{equation}\label{AnisoSecondFFDef}
\restr{S_F}{x} \daA{T_x\Sigma}{T_x\Sigma},\ \restr{S_F}{x}:= \restr{A^F}{\nu_\Sigma(x)} \circ \restr{d\nu_\Sigma}{x}.
\end{equation}
If $\Sigma$ has diagonal anisotropic second fundamental form, then up to rescaling, $\Sigma$ is the Wulff shape. 
\end{teo}
The stability properties for Theorem \ref{AnisoSecondFFTheorem} are almost completely unclear, even at a qualitative level. In this paper we give an answer to this issue, proving the anisotropic sharp counterpart of the isotropic result shown in \cite[Theorem 1.1]{Gioffre2016}. To this aim, for a given hypersurface $\Sigma$, we will denote the tensor $\Sdot_F$ as
$$\Sdot_F:=S_F-\frac{H_F}{n}g,\qquad \mbox{where} \qquad H_F:=\tr_g (S_F).$$
where $g := \delta_{|\Sigma}$ and $\delta$ is the flat metric in $\R^{n+1}$. 

The main result of this paper is the following:
\begin{teo}\label{MainThm}
 Let $n \ge 2$, $p \in (1, \, \infty)$ be given, and let $F$ be an elliptic integrand. 
 
There exists $\delta_0=\delta_0(n,p,F)>0$ such that, if $\Sigma$ is a closed, convex hypersurface in $\R^{n+1}$ satisfying the two conditions 
 \[
 \Per(\Sigma) = \Per (\W), \qquad \mbox{and} \qquad  \int_\Sigma \abs{\Sdot_F}^p \, dV \le \delta_0,
 \]
 then there exist a smooth parametrization $\psi \daA{\W}{\Sigma}$ and a vector $c_0 \in \R^{n+1}$ satisfying the following estimate:
 \begin{equation}\label{MainEstimate}
 \norm{\psi - \id - c_0}_{W^{2, \, p}(\W)} \le C(n,p,F) \norm{\Sdot_F}_{L^p(\Sigma)},
 \end{equation}
 where $C>0$ depends only on $n$, $p$ and $F$.
 
 If $p \in (1, \, n]$, the pinching condition can be dropped.
\end{teo}

\section{Notation, preliminaries and strategy of the proof}
Throughout the paper, we will use the following notation:
\begin{center}
\begin{tabular}{ll}
$\langle \cdot,\cdot \rangle$ & Euclidean scalar product in $\R^{n+1}$. \\
$\langle \cdot,\cdot \rangle_{L^2}$ & scalar product in $L^2$. \\
$d_{\hd}$ & Hausdorff distance \\
$\esse^n$ & standard sphere in $\R^{n+1}$. \\
$\Sigma$ & closed, $n$-dimensional hypersurface in $\R^{n+1}$. \\
$\delta$  & standard metric in $\R^{n+1}$. \\
$\sigma$ & standard metric on $\esse^n$. \\
$\omega$ & restriction of $\delta$ to $\W$.   \\
$g$ & restriction of $\delta$ to $\Sigma$.   \\
$h$ & second fundamental form for $\Sigma$.  \\
$B^g_r(x)$ & geodesic ball in  $\Sigma$ centred in $x$, of radius $r$. \\
$B^k_r(x)$ & ball in  $\erre^k$ centred in $x$, of radius $r$ (when $x=0$, we write $B^k_r$). \\
$\partial$ & usual derivative in $\R^{n+1}$. \\
$D$ & Levi-Civita connection associated to $\esse^n$. \\
$\nabla$ & Levi-Civita connection associated to $\Sigma$ or to $\W$. \\
 \end{tabular}
\end{center}
If not diffrently specified, along the paper $\nu$ will denote the normal vector field associated to $\W$.

Given a bounded hypersurface $\Sigma$, for every function $f:\Sigma\to \erre$, we will call its mean the value 
$$\overline{f}=\frac{1}{\Per(\Sigma)}\int_{\Sigma}f.$$

For every Lebesgue-measurable set $A$, we will denote with $|A|$ the Lebesgue measure of $A$.
Given two sets $A,B\subset \R^{n+1}$, the set 
\[
A \Delta B := A \setminus B \cup B \setminus A
\] 
is called symmetric difference of $A$ and $B$.

A measurable set $E \subset \erre^{n+1}$ is said to be a set of finite perimeter if the distributional gradient $D\chi_E$ of the characteristic function of $E$ is an $(n+1)$-valued Borel measure on $\erre^{n+1}$ with total variation $|D\chi_E|(\erre^{n+1}) < \infty$.

Let $E$ be a set of finite perimeter in $\erre^{n+1}$. We define the \textit{anisotropic asymmetry index} as 
\begin{equation}\label{Asymmetry}
A(E):= \min_{ x \in \erre^{n+1}}\left \{ \frac{\abs{U_\W \, \Delta \,  (x + rE)}  }{\abs{E}} \colon  \,  \abs{rU_\W} = \abs{E} \right \} ,
\end{equation}
where $U_\W$ is the open set enclosed by $\W$, and the \textit{anisotropic isoperimetric deficit} as 
\begin{equation}\label{Deficit}
\delta(E):= \frac{\PF(\partial E)}{(n+1) \abs{U_\W}^{\frac{1}{n+1}} \abs{E}^{\frac{n}{n+1}}} -1.
\end{equation}
The relation among $A(E)$ and $\delta(E)$ is well studied in the framework of isoperimetric problems. Indeed, the following anisotropic deficit estimate proved in \cite[Thm 1.1]{FigMagPrat} holds: 
\begin{teo}\label{DeficiThm}
Every set $E$ of finite perimeter in $\erre^{n+1}$ satisfies the following inequality:
\begin{equation}\label{DeficitEst}
A(E) \le C(n) \sqrt{\delta(E)}.
\end{equation}
If $|E|=|U_\W|$, then the inequality \ref{DeficitEst} can be written as 
\begin{equation}\label{DeficitEstDue}
A(E) \le C(n) \sqrt{\PF( \partial E) - \PF(\W)}.
\end{equation}
\end{teo}

\subsection*{Strategy of the proof} 
For the reader convenience, we divide the proof of Theorem \ref{MainThm} in four sections:
\begin{itemize}
\item[Section \ref{OT}.] The main result of this section is Theorem \ref{MeanMainThm}, in which for every convex, closed hypersurface $\Sigma$, we estimate the oscillation of $S_F$ with $\Sdot_F$:  namely we show that
$$\min_{\lambda \in \erre} \norm{S_F - \lambda \Id}_{L^p(\Sigma)} \le C \norm{\Sdot_F}_{L^p(\Sigma)}.$$
\item[Section \ref{Qual}.] We prove Theorem \ref{QualitativeAppThm}, which is  a qualitative stability version of Theorem \ref{AnisoSecondFFTheorem}, showing that an hypersurface with anisotropic second fundametal form close in $L^p$-norm to the identity must be close to the Wulff shape with respect to Hausdorff distance.
\item[Section \ref{Quan}.] We linearise the expression found in Section \ref{OT} and we obtain a linear equation with some error terms. We study the operator arising from it and in Theorem \ref{MainEstimateb} we find a proper estimate depending on the center of $\Sigma$.
\item[Section \ref{Conclusion}.] We properly center $\Sigma$ and remove the error terms, obtaining Theorem \ref{MainThm}.
\end{itemize}

\section{The oscillation theorem}\label{OT}
In this section we will often assume that the hypersurfaces satisfy
\[
\Per(\Sigma)=1.
\]
Nevertheless, since all the quantities we consider have nice rescaling properties, all the results still hold for hypersurfaces with arbitrary isotropic surface energy. 

The main theorem of this section is the following:

\begin{teo}\label{MeanMainThm}
Let $n \ge 2$, $p \in (1, \, \infty)$ and $\delta_0>0$ be given. Let $F \daA{\esse^n}{(0, \, +\infty)}$ be an elliptic integrand. There exists a constant $C$ depending only on $n$, $p$, $\delta_0$ and $F$ such that the following holds. 

If $\Sigma$ is a closed and convex hypersurface in $\erre^{n+1}$ which satisfies
\begin{itemize}
\item[$(i)$] the perimeter condition: $\displaystyle \Per(\Sigma)=1$
\end{itemize}
and
\begin{itemize}
\item[$(ii)$] the pinching condition: $\displaystyle \int_\Sigma \abs{\Sdot_F}^p \le \delta_0 $,
\end{itemize}
then the following estimate is satisfied:
\begin{equation}\label{MeanMainEstimatea}
\min_{\lambda \in \erre} \norm{S_F - \lambda \Id}_{L^p(\Sigma)} \le C \norm{\Sdot_F}_{L^p(\Sigma)}.
\end{equation}
If $p \in (1, \, n]$, the pinching condition can be dropped.
\end{teo}
The proof of Theorem \ref{MeanMainThm} will follow from the following three propositions: 
\begin{prop}\label{DanielConvexity}
Let $n \ge 2$, $p \in (1, \, n]$ and $\Sigma$ be a closed hypersurface in $\erre^{n+1}$, which is the boundary of a convex set $U$, and satisfies:
\begin{itemize}
\item[$(i)$] the perimeter condition: $\displaystyle \Per(\Sigma)=1 $
\end{itemize}
and
\begin{itemize}
\item[$(ii)$] the pinching condition: $\displaystyle \int_\Sigma \abs{h}^p \le \delta_0 $.
\end{itemize}
Then there exist two positive radii $0<r<R$ depending only on $n$, $p$ and $c_0$ and a vector $x \in \erre^{n+1}$ such that 
\begin{equation}\label{Ballinclusion}
B^{n+1}_r(x) \subset U \subset B^{n+1}_R(x).
\end{equation}
\end{prop}

\begin{prop}\label{SecondFFControl}
Let $n \ge 2$, $p \in (1, \, n]$, $F$ be an elliptic integrand and  $\Sigma$ be a convex, closed hypersurface in $\erre^{n+1}$, satisfying $\Per(\Sigma)=1$. There exist two positive constants $c_1$ and $c_2$, depending only on $n$, $p$ and $F$, such that 
\begin{equation}\label{SecondFFControlEq}
\norm{h}_{L^p(\Sigma)} \le c_1 \norm{S_F}_{L^p(\Sigma)} \le c_2 \coup*{ 1 + \norm{\Sdot_F}_{L^p(\Sigma)} }.
\end{equation}

\end{prop}

\begin{prop}\label{GraphChartThm}
Let $n \ge 2$ and $p \in (1, \, \infty)$ be given. Let $\Sigma$ be a (not necessarily convex) hypersurface in $\erre^{n+1}$. Assume there exists $L>0$ and a graph chart  
\begin{equation}\label{GraphChart}
\phi \daA{B^n_r}{\Sigma},\ \phi(x) = G (x, \, f(x) )
\end{equation}
where $f$ is a Lipschitz function satisfying 
$$\Lip(f) \le L$$ 
and $G \daA{\erre^{n+1}}{\erre^{n+1}}$ is an affine transformation obtained composing a translation and a rotation. Define $q:=\phi(0)$.

Then, for every $0 < \rho \le r$, the geodesic ball $B^g_\rho(q)$ satisfies the inclusion 
\begin{equation}\label{GeoBallinclusion}
\phi\coup*{B^n_{\frac{1}{1+L} \rho}} \subset B^g_\rho(q) \subset \phi \coup*{B^n_\rho}.
\end{equation}
Moreover, there exist a constant $C=C(n, \, p, \, r, \, L)>0$ and a real number $\lambda_q$ such that 
\begin{equation}\label{LocalHFEstimate}
\norm{H_F - \lambda_q}_{L^p \coup{B^g_{r/2}(q)}} \le C \norm{\Sdot_F}_{L^p(\Sigma)}.
\end{equation}
\end{prop}
\begin{prop}\label{Local}
Let $(\Sigma, \, g)$ be a closed manifold with unitary isotropic surface energy. Suppose there exists $\gu \in C^\infty(\Sigma)$, $\rho>0$ and $\beta \geq 0$ such that for every $x \in \Sigma$ the following local estimate is satisfied:
\begin{equation}\label{LocalDue}
\norm{\gu - \lambda(x) }_{L^p_g(B_{r}(x))} \le  \beta,
\end{equation}
with $\lambda(x)\in \R$ depending on $x$, and $r \le 2\rho$. Then $\gu$ satisfies the global estimate:
\[
\norm{\gu - \lambda_0 }_{L^p_g(\Sigma)} \le C \beta,
\]
where $\lambda_0 \in \erre$ and $C = C(n, \, p, \, \rho)$ is a positive constant. 
\end{prop}
Proposition \ref{DanielConvexity} can be found in \cite[Prop. 2.7]{Daniel}. Proposition \ref{SecondFFControl} and Proposition \ref{GraphChartThm} are new results and will be proved at the end of this section. Proposition \ref{Local} is proved in \cite[Lemma 3.2]{Gioffre2017}. 

We give now the proof of Theorem \ref{MeanMainThm} using the previous propositions:

\begin{proof}[Proof of Theorem \ref{MeanMainThm}]
Patching together Proposition \ref{DanielConvexity} and Proposition \ref{SecondFFControl} and properly centring $\Sigma$, we are able to find two radii $0<r<R$ depending only $n$, $p$, $F$ and $\delta_0$\footnote{Proposition \ref{SecondFFControl} actually works only for $1<p \le n$. However, the perimeter condition and the H\"older inequality ensure that the super-critical case implies an $L^n$ bound of $\Sdot_F$. }, such that 
\begin{equation}\label{CentredBallinclusion}
B^{n+1}_r \subset U \subset B^{n+1}_R,
\end{equation}
where $U$ is the open set enclosed by $\Sigma$.

We claim the following:
\begin{equation*}
\begin{split}
\mbox{\textbf{CLAIM}:} \quad &  \mbox{For every $q \in \Sigma$ there exists a graph chart}\\
& \phi_q\daA{x \in B^n_r}{G_q(x,f_q(x)) \in \Sigma},\mbox{\, \, as in \eqref{GraphChart},}\\
&\mbox{such that $f_q$ satisfies $\Lip(f_q,B^n_{r/2}) \le \frac{4(R-r)}{r}  =: L(n, \,p, F, \,  \, \delta_0)$. }
\end{split}
\end{equation*}
Assuming the claim is true and plugging it in Proposition \ref{GraphChartThm}, we deduce that for every $q \in \Sigma$ there exist $\lambda_q$ and $C$ depending only on $n$, $p$, $F$ and $\delta_0$ such that \eqref{LocalHFEstimate} holds. Moreover, from inclusion \eqref{GeoBallinclusion} we obtain a constant $\rho_0=\rho_0( \, p, \, F, \, \delta_0)>0$ such that 
\begin{equation}\label{VolumeEstimate}
\Per(B^g_{r_0}(q)) \ge \rho_0.
\end{equation}
Ffrom Proposition \ref{Local} we obtain a real number $\lambda_0$ and a constant $C=C(n, \, p, \, F, \, \delta_0)$ such that 
\begin{equation}\label{controllo}
\norm{H_F - \lambda_0}_{L^p(\Sigma)} \le C \norm{\Sdot_F}_{L^p(\Sigma)}.
\end{equation}
Consequently, we can estimate 
\begin{equation}\label{usata}
\begin{split}
\inf_{\lambda \in \erre} \norm{S_F - \lambda \Id}_{L^p(\Sigma)} 
&\le \norm*{S_F - \frac{\lambda_0}{n} \Id}_{L^p(\Sigma)} = \norm*{\Sdot_F + \frac{1}{n}(H_F - \lambda_0) \Id}_{L^p(\Sigma)} \\
&\le \norm{\Sdot_F}_{L^p(\Sigma)} + C(n,p) \norm{H_F - \lambda_0}_{L^p(\Sigma)} \\
&\le C(n, \, p, \, F, \, \delta_0) \norm{\Sdot_F}_{L^p(\Sigma)},
\end{split}
\end{equation}
which is the desired estimate \eqref{MeanMainEstimatea}.

Now we show how to get rid of the dependence on $\delta_0$, under the assumption that $p \in (1, \, n]$. Indeed, in this regime we have the following dichotomy: either $\norm{\Sdot_F}_p \leq \delta_0\leq 1$ or we can assume that $\norm{\Sdot_F}_p > 1$. In the first case, since $\delta_0$ can be set equal to $1$, we have proved that there exists a constant $C=C(n, \, p, \, F)$ such that 
\[
\min_{\lambda \in \erre} \norm{S_F - \lambda \Id}_{L^p(\Sigma)} \le C \norm{\Sdot_F}_{L^p(\Sigma)}.
\]
Instead, in the latter case $\norm{\Sdot_F}_p > 1$, we obtain the following estimate chosing as particular competitor $\lambda=0$ and applying Proposition \ref{SecondFFControl}:
\begin{align*}
\min_{\lambda \in \erre} \norm{S_F - \lambda \Id}_{L^p(\Sigma)} 
&\le  \norm{S_F}_{L^p(\Sigma)} \overset{\eqref{SecondFFControlEq}}{\le} C(n,p,F) \coup*{1 + \norm{\Sdot_F}_{L^p(\Sigma)} } \\& \le C(n,p,F) \coup*{\frac{1}{\norm{\Sdot_F}_{L^p(\Sigma)} } + 1 }\norm{\Sdot_F}_{L^p(\Sigma)} \\
& \le C(n,p,F) \norm{\Sdot_F}_{L^p(\Sigma)}
\end{align*}
where the constant $C$ does not depend on $\delta_0$.

In order to conclude the proof, we are just left to prove the claim.

 Let $q \in \Sigma$ be given. We can assume $q= - \abs{q} e_{n+1}$, where we have denoted by $\set{e_i}_{i=1}^{n+1}$ the standard basis for $\erre^{n+1}$. Let $\pi$ be the affine hyperplane parallel to $\Span\set{e_1, \, \dots e_n}$ containing $q$. Then we can easily find $\phi_q \daA{B^n_r}{\Sigma}$ given by the composition of the translation $B^n_r \to B^n_r + q \subset \pi $ and the projection from $\pi$ to $\Sigma$. Clearly $\phi_q$ is a graph chart, and we can write $\phi_q(x):=(x, \, f_q(x))$ for every $x \in B^n_r$. By assumption, we have that $f_q$ is a convex function, satisfying $f_q(0)=q$  and $r \le f_q \le R$. Hence, as shown in \cite{Convex}, the Lipschitz constant of $f$ satisfies the estimate: 
\[
\Lip\coup*{f_q, \, B^n_{ \faktor{r}{2} }} \le \frac{4(R-r)}{r},
\]
which gives us the claim and completes the proof.
\end{proof}

Before to conclude this section, we prove Proposition \ref{SecondFFControl} and Proposition \ref{GraphChartThm}. 

\subsection{Proof of Proposition \ref{SecondFFControl}}
Firstly, we notice that a closed convex hypersurface has non-negative anisotropic principal curvatures. Although this result seems to be known, we did not find its proof in the literature. 
\begin{lemma}\label{AnisoNonNegative}
Let $\Sigma$ be a closed hypersurface, and let $\set{\kappa_1, \, \dots \kappa_n}$ be the spectrum of $S_F$. If $\Sigma$ is convex, then $\kappa_i \ge 0$ for every $i=1,\dots,n$.
\end{lemma}
\begin{proof}
 We recall that $(S_F)^i_j = (A^F)^i_k h^k_j$, where $A^F$ is positive definite by hypothesis and $h$ is non-negative definite by convexity (see \cite[Prop. 3.2]{Daniel}). 

Let $(A^F)^{\frac{1}{2}}$ be the square root of $A^F$. By standard linear algebra, $(A^F)^{\frac{1}{2}}$ exists, is unique and symmetric, with positive eigenvalues. Then we find
\[
A^F h =(A^F)^{\frac{1}{2}} \coup*{ (A^F)^{\frac{1}{2}} h (A^F)^{\frac{1}{2}} } (A^F)^{-\frac{1}{2}}.
\]
By this simple decomposition, we deduce that $A^F h$ has the same eigenvalues of $(A^F)^{\frac{1}{2}} h (A^F)^{\frac{1}{2}}$. This completes the proof: indeed, for every vector $v\in \R^n$, since $h$ is non-negative definite, we can compute 
\[
\coup*{ (A^F)^{\frac{1}{2}} h (A^F)^{\frac{1}{2}} [v], \, v } = \coup*{  h (A^F)^{\frac{1}{2}} [v], \, (A^F)^{\frac{1}{2}}[v] } \ge 0,
\]
which is the thesis.
\end{proof}
If we look carefully at the proof of Lemma \ref{AnisoNonNegative}, we can also notice that we have found the existence of a constant $c_1=c_1(n, \, p, \, F)$ such that 
\[
\norm{h}_{L^p(\Sigma)} \le c_1 \norm{S_F}_{L^p(\Sigma)}.
\]
In order to conclude the proof of \eqref{SecondFFControl}, we just have to focus on showing the remaining inequality
\begin{equation}\label{SecondFFControlEq2}
\norm{S_F}_{L^p(\Sigma)} \le c_2 \coup*{ 1 + \norm{\Sdot_F}_{L^p(\Sigma)} }.
\end{equation}
This follows by generalizing the isotropic result shown in \cite{Daniel}.
Firstly, we notice that, for every couple of indexes $(i,j)$, we have:
\begin{align*}
\coup*{ \int_\Sigma \abs{\kappa_i - \kappa_j}^p }^{\frac{1}{p}} 
&\le   \coup*{ \int_\Sigma \abs{\kappa_i - H_F}^p }^{\frac{1}{p}}  + \coup*{ \int_\Sigma \abs{\kappa_j - H_F}^p }^{\frac{1}{p}} \le c(n, \, p) \norm{\Sdot_F}_p.
\end{align*}
Consequently, we can estimate: 
\begin{equation}\label{buono}
\begin{split}
\norm{S_F}_{L^p(\Sigma)}
&\le \coup*{ \int_\Sigma \coup*{\sum_{i=1}^n \kappa_i^2 }^\frac{p}{2} }^{\frac{1}{p}} \le \coup*{ \int_\Sigma \coup*{\sum_{i=1}^n \kappa_i }^p }^{\frac{1}{p}} \\
&= \coup*{ \int_\Sigma \coup*{ n \kappa_1 + \sum_{i=2}^n ( \kappa_i - \kappa_1) }^p }^{\frac{1}{p}}\\
& \le c(n, \, p) \coup*{ \coup*{ \int_\Sigma \kappa_1^p }^{\frac{1}{p}} + \norm{\Sdot_F}_p } \\
&\le c(n, \, p) \coup*{ \coup*{ \int_\Sigma \kappa_1^n }^{\frac{1}{n}} + \norm{\Sdot_F}_p } \\
&\le  c(n, \, p) \coup*{ \coup*{ \int_\Sigma \det S_F }^{\frac{1}{n}} + \norm{\Sdot_F}_p } .
\end{split}
\end{equation}
We observe furthermore that the integral of the determinant of $S_F$ does not depend on $\Sigma$ but only on $n$, $p$ and $F$. Indeed, by the area formula (see \cite[Chap. 2.10]{AFP} we find  
\[
\int_{\Sigma} \det S_F = \int_\Sigma \restr{\det A_F }{\nu} \det d\nu = \int_{\esse^n} \det A_F = c(n, \, p, \, F).
\]
Plugging the previous equality in \eqref{buono}, we deduce \eqref{SecondFFControlEq2} and conclude the proof.

\subsection{Proof of Proposition \ref{GraphChartThm}}
First of all we point out that for any convex, closed hypersurface $\Sigma$, the normal vector field $\nu$ is a diffeomorphism between $\Sigma$ and $\esse^n$. By the very definition of $\nu$, we recall that $T_{\nu(x)} \esse^n = \langle \nu(x) \rangle^{\perp} = T_x \Sigma$, and therefore at the point $x$ we find the equality
\begin{equation}\label{gAndsigma}
g(x) := \restr{\delta}{ T_{\nu(x)} \esse^n  }  = \restr{\delta}{ T_x \Sigma }  =: \sigma (\nu(x)).
\end{equation}
Equality \ref{gAndsigma} allows us to raise and lower indexes in the expressions involving both tensors defined on $\esse^n$ and tensors defined on $\Sigma$. This simple observation will be repeatdly used in the next computations.
We first need to prove the following lemma: 
\begin{lemma}\label{AnisoCodazzi}
Let $\Sigma$ be an oriented hypersurface in $\erre^n$. Then, the following expression holds
\begin{equation}\label{AnisoCodazziEq}
\nabla H_F = \divv_g S_F
\end{equation}
\end{lemma}
\begin{proof}
For the sake of simplicity, throughout this proof we will omit the depndence of $A^F$ and $S_F$ on $F$, denoting the tensors just with $A$ and $S$. Firstly, we notice that $A$ is a Codazzi tensor, namely it satisfies the symmetry
\[
D_i A_{jk} = D_j A_{ik}.
\]
Using this property, we can compute
\begin{align*}
\divv_g S_{k} 
&= \nabla_i S^i_k = \nabla_i \coup*{\restr{A^i_p}{\nu} h^p_k} =  \nabla_i \coup*{\restr{A^i_p}{\nu} } h^p_k + \restr{A^i_p}{\nu} \nabla_i  h^p_k \\
&= \restr{D_q A^i_p}{\nu} h^q_i h^p_k + \restr{A^i_p}{\nu} \nabla_i  h^p_k = \restr{D_p A^i_q}{\nu} h^q_i h^p_k + \restr{A^i_p}{\nu} \nabla^p  h_{ik} \\
&= \nabla_k \coup*{ \restr{A^i_p}{\nu} h^p_i } = \nabla_k H_F ,
\end{align*}
as desired.
\end{proof}
Proposition \ref{GraphChartThm} now follows simply by expanding equation \eqref{AnisoCodazziEq} in graph charts. Indeed, let $\phi$ be a graph chart, that is 
\[
\phi \daA{B^n_r}{\Sigma},\ \phi(x) = (x, \, f(x)).
\] 
With a straightforward computation, one can find the equality
\[
\Gamma^k_{ij} = v^k \, h_{ij}, \qquad \mbox{ where } v= \frac{\partial f}{\sqrt{1 + \abs{\partial f}^2}}.
\] 
Using the previous expression, we obtain that: 
\begin{align*}
\divv_g S_{ k} 
&= \nabla_i S^i_k  = \partial_i S^i_k + \Gamma_{ip}^i S^p_k - \Gamma^p_{ik} S^i_p = \partial_i S^i_k + v^i h_{ip} S^p_k - v^p h_{ik} S^i_p \\ 
&=   \partial_i S^i_k + v^i h_{ip} A^p_q h^q_k - v^p h_{ik} A^q_p h^i_q = \partial_i S^i_k + v^i A^{pq} \coup* {h_{ip}  h_{qk} - h_{pk}  h_{qi} } \\
&= \partial_i S^i_k + v^i h_{ip}  h_{qk} \coup*{A^{pq} - A^{qp}} = \partial_i S^i_k. 
\end{align*}
This computation shows that in graph chart, equation \eqref{AnisoCodazziEq} can be written in the simple form:
\begin{equation}\label{AnisoCodazziChart}
\partial_k H_F = \partial_i S^i_k.
\end{equation}
By \eqref{AnisoCodazziChart} we immediately obtain 
\begin{equation}\label{AnisoCodazziChartDot}
\partial_k H_F = \frac{n}{n-1} \partial_i \Sdot^i_k \qquad \mbox{ in } B^n_r.
\end{equation}
Equation \eqref{AnisoCodazziChartDot} has been intensively studied in the literature and we can conclude from it (see for instance \cite{GilTrud}) that there exists a $\lambda$ such that 
\begin{align*}
\norm{H_F - \lambda}_{L^p_\delta(B^n_{\faktor{r}{2}})} 
&\le C(n, \, p, \, r) \norm{ \Sdot_F }_{ L^p_\delta (B^n_r) }\\
&\le C(\Lip(f), \, n, \, p, \, r) \norm{ \Sdot_F }_{ L^p_g (B^n_r) }  
\le C \norm{\Sdot_F}_p,
\end{align*}
which is the desired estimate in the flat volume $dx$. We now prove inclusion \eqref{GeoBallinclusion}. The proof is essentially taken from \cite[Lemma 1.7]{Daniel}, therefore we just give a sketch of it. Let us denote by $d_g$ the geodesic distance associated to the metric $\phi^* g$ in the ball $B^n_r$. We prove that, for every $0<\rho< r$, the following inclusion holds:
\begin{equation}\label{GeoInclusion}
\partial B^g_\rho \subset \overline{B}^n_\rho \setminus B^n_{\frac{\rho}{1 + L}}.
\end{equation}
Indeed, let $y \in \partial B^g_\rho $. We have $d_g(0, \, y) = \rho$ by definition, and moreover the following inequality holds:
\[
d_g(0, \, y) \ge \abs{(y, \, f(y) - (0, \, 0)} \ge \sqrt{\abs{y}^2 + f(y)^2} \ge \abs{y},
\]
which proves $\abs{y} \le \rho$. In order to prove the other inclusion, let us define $\gamma(t):= (ty, \, f(ty))$. Recalling that $f$ is $L$-Lipschitz, we obtain 
\begin{align*}
d_g( 0, \, y) 
&\le \int_0^1 \abs{\dot{\gamma}(t)} \, dt = \int_0^1 \sqrt{\abs{y}^2 + \coup*{ \restr{\partial f}{ty}, \, y  }^2 } \, dt \\
&\le \int_0^1 \sqrt{\abs{y}^2 \coup*{1 +  \abs{\restr{\partial f}{ty}}^2}} \, dt \le (1+L)\abs{y},
\end{align*}
which proves the other inclusion. We complete the theorem showing how to pass from the flat volume $dx$ to the  measure $dV$ associated to $(\Sigma, \, g)$. Indeed it is easy to show that in graph charts we have the expression
\[
dV = \sqrt{1 + \abs{\partial f}^2} \, dx,
\]
hence we obtain 
\[
dx \le dV \le (1 + L) \, dx
\]
and the thesis.

\section{The qualitative result}\label{Qual}
In this section we show the first qualitative, closeness theorem of this paper. Although this result is not optimal, it is the key point for our further study.
\begin{teo}\label{QualitativeAppThm}
Let $n \ge 2$, $p \in (1, \, \infty)$ be given, and let $F$ be an elliptic integrand. For every $0<\epsilon<1$ there exists $0<{\delta=\delta(n, \, p, \, F, \, \epsilon)} < 1$ with the following property.

If $\Sigma$ is a closed, convex hypersurface satisfying 
\begin{itemize}
\item[$(i)$] the perimeter condition: $\displaystyle \Per(\Sigma) = \Per(\W) $ 
\end{itemize}
and 
\begin{itemize}
\item[$(ii)$] the pinching condition:  $\displaystyle \int_\Sigma \abs{\Sdot_F}^p \, dV \le \delta$ ,
\end{itemize}
then there exists $c \in \erre^n$ such that 
\begin{equation}\label{C0CloseEqq}
d_{\hd}\coup*{\Sigma, \, \W + c} \le \epsilon
\end{equation}
\end{teo}

\begin{proof}
We argue by contradiction ad assume there exist $\epsilon_0>0$ and a sequence of closed, convex hypersurfaces $\set{\Sigma^k}_{k \in \N}$ satisfying 
\begin{itemize}
\item[$(i)'$]  $\displaystyle \Per(\Sigma^k) = \Per(\W) $ 
\item[$(ii)'$] $\displaystyle \lim_k \norm{\Sdot^k_F}_{L^p(\Sigma^k)}=0$
\item[$(iii)'$] $\displaystyle d_{\hd}\coup*{\Sigma^k, \, \W + c } \ge \epsilon_0$ for every $c \in \R^{n+1}$ and $k \in \N$.
\end{itemize}
We will denote by $h^k$,$S_F^k$, $\Sdot^k_F$, $H_F^k$ and $\nu^k$ respectively the second fundamental form, the anisotropic second fundamental form, its trace-free part, the anisotropic mean curvature and the normal vector field associated to the hypersurfaces $\Sigma^k$.
We notice that from condition $(i)'$, condition $(ii)'$ and Proposition \ref{DanielConvexity}, we are able to find two radii $0<r<R$, depending only on $n$, $p$ and $F$\footnote{In the super-critical case we can assume that up to extract a subsequence, every $\norm{\Sdot^k_F}_p$ is bounded by $1$, hence removing the depence on $c_0$ for the qualitative argument. }, such that inclusion  \eqref{Ballinclusion} holds. We centre every $\Sigma^k$ so that \eqref{CentredBallinclusion} holds.  For every $k$, let $U^k$ be the convex bounded set enclosed by $\Sigma^k$. From the Blaschke's selection theorem (see \cite[Thm. 1.8.6]{Schneider}), we infer that, up to estract a subsequence, we can assume that $\overline{U}^{(k)} \to V$ in the Hausdorff distance $d_{\hd}$. From inclusion \eqref{CentredBallinclusion} we infer that the volumes $U^k$ do not converge to $0$, hence $V$ has positive measure and necessarily it has the form $V = \overline{U}$ for some $U$ open, bounded and convex set.

Let $\Sigma$ be the boundary of $U$. From the discussion above, we easily notice that $\Sigma^k$ converges to $\Sigma$ in the Hausdorff distance. Plugging this information in $(iii)'$, we deduce that $\displaystyle d_{\hd}\coup*{\Sigma, \, \W + c } \ge \epsilon_0$ for every $c \in \R^{n+1}$. If we show that $\Sigma$ is a Wulff shape, we obtain the desired contradiction. 

As shown in equation \eqref{controllo}, there exists a sequence $(\lambda_k)_k\subset \R$ such that
\begin{equation}\label{stima}
\|H_F^k-\lambda_k\|_{L^p(\Sigma^k)}\leq  C\|\Sdot_F^k\|_{L^p(\Sigma^k)}\to 0.
\end{equation}
Moreover, by Proposition \ref{SecondFFControl} we know that for $k$ large enough
$$\|H_F\|_{L^p(\Sigma^k)}\leq c_2 \coup*{ 1 + \norm{\Sdot_F}_{L^p(\Sigma)} } \leq 2c_2.$$
It follows that 
$$|\lambda_k|=\frac{\|\lambda_k\|_{L^p(\Sigma^k)}}{\Per(\Sigma^k)^\frac{1}{p}}\leq C(\|H_F^k\|_{L^p(\Sigma^k)}+\|H_F^k-\lambda_k\|_{L^p(\Sigma^k)}) \leq C.$$
We conclude that there exists $\lambda \in \erre$ such that, up to subsequences, $\lambda_k \to \lambda$. Therefore, we can assume we are given a sequence $\{ \Sigma^k \}$ satisfying the following properties.
\begin{itemize}
\item[$(i)''$] $\Sigma^k= \partial U^k$, with $U$ being a convex, open, bounded set satisfying $B_r \subset U^k \subset B_R$ .
\item[$(ii)''$] There exists $\Sigma= \partial U$ such that $d_{\hd}\coup*{\Sigma^k, \, \Sigma} \to 0$.
\item[$(iii)''$] $\norm{S_F^{k} - \lambda \Id}_{L^p(\Sigma^k)} \to 0$ for some $p \in (1, \, \infty)$.
\end{itemize}
We show how these three conditions imply $\Sigma$ to be the Wulff shape. Firstly, condition $(i)''$ allows us to give the following radial parametrization for every $k$:
\begin{equation}\label{RadialPar}
\psi_k \daA{\esse^n}{\Sigma^k}, \qquad \psi_k(x) = e^{f_k(x)} \, x.
\end{equation}
Clearly, $\psi_k$ is a smooth parametrization for every $k$. By condition $(i)''$ and convexity, it is easy to see that every $f_k$ satisfies 
\begin{equation}\label{PropertiesF}
\log(r) \le f_k \le \log (R), \qquad \Lip(f_k) \le L=L(r, \, R).
\end{equation}
Moreover, by condition $(ii)''$, we find that $f_k$ converges in $C^0$ to a function $f$ satisfying \eqref{PropertiesF} and such that 
\[
\psi \daA{\esse^n}{\Sigma}, \qquad \psi(x) = e^{f(x)} \, x
\]
is a Lipschitz parametrisation for $\Sigma$. We can improve the regularity of $f$. Indeed, by condition $(iii)''$, we can write 
\[
S_F^k = \lambda \Id + \mathcal{R}_k, \qquad \mbox{ where } \norm{\mathcal{R}_k}_{L^p(\Sigma^k)} \to 0.
\]
We know that $S_F^k=A^F \circ d\nu^k$, where $A^F$ is the smooth $2$-covariant tensor defined on the sphere as in \eqref{AF} and such that $0<A^F$. Multiplying by the inverse of $A^F$ and taking the $L^p$ norm, we obtain the estimate 
\begin{equation}\label{EllipticEstimate}
\norm{h^k}_{L^p(\Sigma^k)} \le \norm{(A^F)^{-1}}_{C^0} \coup*{\lambda + \norm{\mathcal{R}_k}_{L^p(\Sigma^k)}}.
\end{equation}
We exploit the fact that we have $\psi^k$ as a global chart. The computations shown in \cite[Lemma 3.1]{Gioffre2016} give us the following expression for $h^k$ in this chart:
\begin{equation}\label{hExpression}
(h^k)^i_j= \frac{e^{-f_k}}{\sqrt{1 + \abs{Df_k}^2}} \coup*{ \delta^i_j - D^i D_j f_k + \frac{D^i f_k}{1 + \abs{D f_k}^2} D^p D_j f_k \, D_p f_k }.
\end{equation}
Plugging expression \eqref{hExpression} into \eqref{EllipticEstimate}, we easily obtain 
\begin{equation}\label{EllipticEstimateDue}
\sup_k \norm*{\divv_\sigma \coup*{ \frac{D f_k}{\sqrt{1 + {Df_k}^2}} }}_{L^p(\esse^n)} < + \infty.
\end{equation}
Since $f_k$ satisfies \eqref{PropertiesF}, standard regularity theory gives
\begin{equation}
\sup_k \norm{f_k}_{W^{2, \, p}(\esse^n)} < + \infty.
\end{equation}
So the limit $f$ must be in $W^{2, \, p}(\esse^n)$ and hence the limit $\Sigma$ is a rough hypersurface with $W^{2, \, p}$ regularity. Since the second weak derivatives have meaning, by condition $(iii)''$ we know that the (weak) anisotropic second fundamental form $S_F$ of $\Sigma$ satisfies weakly the equation 
\begin{equation}\label{Final}
S_F = \lambda \Id.
\end{equation}
Using the expression \eqref{hExpression} for the second fundamental form $h$ of $\Sigma$ and taking the trace, we obtain that $f$ satisfies \eqref{PropertiesF} and the differential equation
\begin{equation}\label{EllipticRegularityTre}
\divv_\sigma \coup*{\frac{Df}{\sqrt{1 + \abs{Df}^2}}} = \lambda \, \tr \restr{A^F}{\nu} + \frac{n e^{-f}}{\sqrt{1 + \abs{D f}^2}} =: h(f, \, Df).
\end{equation}
By standard elliptic regularity theory (see \cite{GilTrud}), we obtain that the hypersurface is smooth, and necessarily equality \eqref{Final} holds in the classical sense. Hence, by Theorem \ref{AnisoSecondFFTheorem}, $\Sigma$ must be the Wulff shape.
\end{proof}

\section{The quantitative result}\label{Quan}
Before stating the main theorem of this section, we define the parametrization we will work with and its properties. 
Let $\Sigma$ be a closed, convex hypersurface in $\erre^{n+1}$ satisfying the closeness condition \eqref{C0CloseEqq}, and let $B_\epsilon(\W)$ be the tubular neighbourhood associated to $\W$, that is the set
\begin{equation}\label{Tubular}
B_\epsilon(\W):= \set{ z \in \R^{n+1} \mid z= x+ \rho \nu(x), \quad \forall  x \in \W, \, 0 \le \rho <\epsilon   }.
\end{equation}
All the details details needed on the tubular neighbourhood can be found in \cite[Chap. 5]{Hirsch}. We recall that  for $\epsilon$ sufficiently small,  $B_r(\W)$ is an open, bounded set with smooth boundary diffeomorphic to $\W$, for every $r<\varepsilon$. By the very definition, if $\Sigma$ is $\epsilon$ close to $\W$ in the Hausdorff distance, then it is contained in the tubular neighbourhood $B_\epsilon(\W)$. Assume this condition holds: then we can give the following parametrization for $\Sigma$
\begin{equation}\label{RadialPar1}
\psi \daA{\W}{\Sigma},\ \psi(x)= x + u(x) \nu(x),  \qquad \mbox{ for some } u \in C^\infty(\W).
\end{equation}
 Clearly $\psi$ is a smooth diffeomorphism. We call $\psi$ \textit{radial parametrization of} $\Sigma$ and $u$ \textit{radius} of $\Sigma$. 
 
 We have all the notation to state the main result of this section:
\begin{teo}\label{MainEstimateb}
Let $\Sigma$ be a closed, convex, radially parametrized hypersurface in $\erre^{n+1}$ as in \eqref{RadialPar1}, satisfying the perimeter condition:
$$\Per(\Sigma) = \Per(\W),$$ 
 and having radius $u$ verifying
 \[
\norm{u}_{C^0} \le \epsilon, \quad \norm{\nabla u }_{C^0} \le C\sqrt{\epsilon} .
\] Then there exists a constant $C=C(n, \, p, \, F)>0$ such that 
 \begin{equation}\label{MainEstimateEq}
 \inf_{c \in \erre^{n+1}} \norm{u - \phi_c}_{W^{2, \, p}(\W)} \le C \coup*{ \norm{\Sdot_F}_{L^p(\Sigma)} + \sqrt{\epsilon} \norm{u}_{W^{2, \, p}(\W)} },
 \end{equation}
 where we have denoted by $\phi_c$ the linear projection of $\nu$ by the vector $c$, that is 
 \begin{equation}\label{PhiDef}
 \phi_c \daA{\W}{\R},\quad \phi_c(y) := \langle c, \, \nu(y)\rangle .
 \end{equation}
\end{teo}
 
Before proving Theorem \ref{MainEstimateb}, we state two computational propositions, whose proof is postponed to Appendix \ref{APPA}: 

\begin{prop}\label{GeomApp}
Let $\Sigma$ be a closed, radially parametrized hypersurface as in \eqref{RadialPar1}, with radius satisfying the inequalities
\[
\norm{u}_{C^0} \le \epsilon, \quad \norm{\nabla u }_{C^0} \le C\sqrt{\epsilon} .
\]
There exists a constant $C=C(\W)$ such that the following inequalities hold:
\begin{align}
\abs{g_{ij} - \omega_{ij} - 2u h_{ij}} &\le C \sqrt{\epsilon}  \coup*{\abs{u} + \abs{\nabla u}} \label{gApp}  \\
\abs{g^{ij} - \omega^{ij} - 2u h^{ij}} &\le C \sqrt{\epsilon}  \coup*{\abs{u} + \abs{\nabla u}} \label{gInvApp} \\
\abs{\det g - \det \omega} &\le C \sqrt{\epsilon} \coup*{ \abs{u} + \abs{\nabla u}} \label{detApp} \\
\abs{\nu_\Sigma - \nu + \nabla u}    &\le  C \sqrt{\epsilon} \coup*{\abs{u} + \abs{\nabla u}} \label{nuApp} \\
\abs{h_\Sigma - h_\W + \nabla^2 u - h^2_\W u} &\le  C \sqrt{\epsilon} \coup*{\abs{u} + \abs{\nabla u} + \abs{\nabla^2 u }} \label{hApp} \\
|S_F(\Sigma)-S_F(\W)+\nabla(A^F\nabla u)-h_\W u|&\le  C \sqrt{\epsilon} \coup*{\abs{u} + \abs{\nabla u} + \abs{\nabla^2 u }} \label{secondaforma} \\
|\overline{H_F(\Sigma)}-\overline{H_F(\W)} |&\leq C \sqrt{\epsilon} \coup*{\abs{u} + \abs{\nabla u} + \abs{\nabla^2 u }} \label{curvatura} .
\end{align}
\end{prop}

 \begin{prop}\label{C1Close}
Let $\Sigma$ be a closed, convex hypersurface in $\erre^{n+1}$ such that $\Sigma \subset B_\epsilon(\W)$.  If $\epsilon>0$ is small enough, there exists $0<C=C(\W)$ such that $\Sigma$ is $C^1$ close to $\W$, that is 
\begin{align}
\norm{u}_{C^0} \le \epsilon \label{AppC0} \\
\norm{\nabla u }_{C^0} \le C\sqrt{\epsilon}. \label{AppC1}
\end{align}
\end{prop}
 
We now have all the tools to prove Theorem \ref{MainEstimateb}:

\begin{proof}[Proof of Theorem \ref{MainEstimateb}]
The idea behind the proof is to perform a proper, quantitative linearisation of $S_F$ for small $u$. 
For every $\eta \in C^\infty(\W)$, we define the $1$-parameter family of diffeomorphisms (provided $t$ is small enough)
\begin{equation}\label{RadialPara}
\psi_t \daA{\W}{\Sigma_t},\ \psi(x)= x + t \eta(x) \nu(x).
\end{equation}
We will denote
$$h(t):=h(\Sigma_t), \qquad \nu_t:=\nu_{\Sigma_t}, \qquad  S_F(t):=S_F(\Sigma_t),  \qquad \mbox{and}\qquad A:=A^F$$
and when the time dependence is omitted, it has to be intended evaluated in $0$.
We observe that
$$\restr{\frac{\partial}{\partial t}}{t=0}h^i_j(t)=-\nabla^i\nabla_j\eta-\eta h^i_kh^k_j, \quad \mbox{and} \quad \restr{\frac{\partial}{\partial t}}{t=0}h_{ij}(t)=-\nabla_i\nabla_j\eta+\eta h_{ik}h^k_j.$$
We have all the notation to define the stability operator
\begin{equation}\label{stabop}
\begin{split}
\tilde{L}[\eta]^i_j:=-\restr{\frac{\partial}{\partial t}}{t=0}(S_F(t))^i_j&=- \restr{\frac{\partial}{\partial t}}{t=0}\left(\restr{A^i_k}{\nu_t}h_j^k(t)\right)\\
&=-\restr{A^i_k}{\nu}\restr{\frac{\partial}{\partial t}}{t=0}\left(h_j^k(t)\right)-\restr{\frac{\partial}{\partial t}}{t=0}\left(\restr{A^i_k}{\nu_t}\right)h_j^k\\
&=\restr{A^i_k}{\nu}\left(\nabla^k\nabla_j\eta+\eta h^k_ph^p_j\right)+\restr{D_p(A^i_k)}{\nu}h_j^k\nabla^p\eta\\
&=\nabla_j(A\nabla \eta)^i+\eta h_j^i.
\end{split}
\end{equation}
Now let us come back to the study of $\Sigma$ parametrized by the map $\psi$ as in \eqref{RadialPar1}.
We first claim that 
\begin{equation}\label{Claimone}
 \norm{S_F(\Sigma)-S_F(\W)+\tilde{L}[u]}_{L^p(\W)}  \le C \sqrt{\epsilon} \norm{u}_{W^{2, \, p}(\W)}.
 \end{equation}
Indeed, from \eqref{stabop} and \eqref{secondaforma} we can estimate
$$|S_F(\Sigma)-S_F(\W)+\tilde{L}[u]|\le  C \sqrt{\epsilon} \coup*{\abs{u} + \abs{\nabla u} + \abs{\nabla^2 u }}.$$
We consequently get \eqref{Claimone}, taking the power $p$ and integrating on $\Sigma$.
From \eqref{Claimone}, we estimate
\begin{equation}\label{poi}
\begin{split}
\norm{\tilde{L}[u]}_{L^p(\W)}  &\le  \norm{S_F(\Sigma)-S_F(\W)]}_{L^p(\W)}+\norm{S_F(\Sigma)-S_F(\W)+\tilde{L}[u]}_{L^p(\W)}\\
&\le \norm{S_F(\Sigma)-S_F(\W)]}_{L^p(\W)}+C \sqrt{\epsilon} \norm{u}_{W^{2, \, p}(\W)}\\
&\le \left\|S_F(\Sigma)- \frac{\overline{H_F(\Sigma)}}{n} \Id\right\|_{L^p(\W)}+\left\|S_F(\W)- \frac{\overline{H_F(\W)}}{n} \Id\right\|_{L^p(\W)}\\
& \quad +C \norm{\overline{H_F(\Sigma)}-\overline{H_F(\W)} }_{L^p(\W)}+C \sqrt{\epsilon} \norm{u}_{W^{2, \, p}(\W)},
\end{split}
\end{equation}
which thanks to the characterization of the Wulff shape $S_F(\W) \equiv \frac{1}{n}\overline{H_F(\W)} \Id$ (see Theorem \ref{AnisoSecondFFTheorem}), reads
\begin{equation}\label{poi1}
\begin{split}
\norm{\tilde{L}[u]}_{L^p(\W)}  &\le \left\|S_F(\Sigma)- \frac{\overline{H_F(\Sigma)}}{n} \Id\right\|_{L^p(\W)}\\
&\quad +C \norm{\overline{H_F(\Sigma)}-\overline{H_F(\W)} }_{L^p(\W)}+C \sqrt{\epsilon} \norm{u}_{W^{2, \, p}(\W)}.
\end{split}
\end{equation}
Using \eqref{curvatura}, we get that 
$$\norm{\overline{H_F(\Sigma)}-\overline{H_F(\W)} }_{L^p(\W)}\leq C \sqrt{\epsilon} \norm{u}_{W^{2, \, p}(\W)},$$
which, plugged in \eqref{poi1}, gives the estimate
\begin{equation}\label{poi2}
\begin{split}
\norm{\tilde{L}[u]}_{L^p(\W)}  &\le  \left\|S_F(\Sigma)- \frac{\overline{H_F(\Sigma)}}{n} \Id\right\|_{L^p(\W)}+C \sqrt{\epsilon} \norm{u}_{W^{2, \, p}(\W)}.
\end{split}
\end{equation}
Now, we recall that by \eqref{usata}
$$  \left\|S_F(\Sigma)- \frac{\overline{H_F(\Sigma)}}{n} \Id\right\|_{L^p(\W)} \leq C(n,p,F) \norm{\Sdot_F}_p,$$
which, together with \eqref{poi2}, gives
$$\norm{\tilde{L}[u]}_{L^p(\W)} \le C \coup*{ \norm{\Sdot_F}_{L^p(\Sigma)} + \sqrt{\epsilon} \norm{u}_{W^{2, \, p}(\W)} }.$$ 
In order to conclude the proof of \eqref{MainEstimateEq}, we are just left to show that
\begin{equation}\label{quasi0}
\inf_{c \in \erre^{n+1}} \norm{u - \phi_c}_{W^{2, \, p}(\W)} \le C\norm{\tilde{L}[u]}_{L^p(\W)}.
\end{equation}
To this aim, since the normal vector field $\nu_\Sigma:\W\to \esse^n$ of $\Sigma$ is a diffeomorphism, we can assume that $u=f\circ \nu_\Sigma$, for some $f \in C^\infty(\esse^n)$. Under this assumption, we can consequently write $\nabla^k u=\restr{D^pf}{\nu}h^k_p$ and obtain
$$\tilde{L}[f\circ \nu]^i_j=\nabla_j(A\nabla u)^i+h_j^if\circ \nu=\nabla_j(\restr{D^if}{\nu})+h_j^if\circ \nu=\restr{D^iD_pf}{\nu}h_j^p+h_j^if\circ \nu.$$
Observe that, since $A$ is invertible, then $\tilde{L}[f\circ \nu]=0$ if and only if 
$$0=A \tilde{L}[f\circ \nu]=A^j_k(\restr{D^iD_pf}{\nu}h_j^p+h_j^if\circ \nu)=D^iD_kf+f\delta_k^i.$$
But we already know that $D^2f+f\sigma=0$ if and only if $f(x)=\langle c, x\rangle$.
We deduce that
\begin{equation}\label{quasi}
\begin{split}\norm{\tilde{L}[u]}_{L^p(\W)} &\ge C \norm{D^2f+f\sigma}_{L^p(\esse^n)}\ge C \norm{\Delta f+nf}_{L^p(\esse^n)}\\
&=\inf_{c \in \erre^{n+1}} \norm{f - \Phi_c}_{W^{2, \, p}(\esse^n)},\end{split}
\end{equation}
where  $\Phi_c \daA{\esse^n}{\R},\ \Phi_c(y) := \langle c, \, y \rangle$ and the last inequality can be found in \cite[Theorem 3.10]{Gioffre2016}.
Since $\nu_\Sigma:\W\to \esse^n$ is a diffeomorphism, we know that
$$\norm{f - \Phi_c}_{W^{2, \, p}(\esse^n)}\geq C \norm{u - \phi_c}_{W^{2, \, p}(\W)}.$$
Plugging this inequality in \eqref{quasi}, we get the desired estimate \eqref{quasi0} and conclude the proof.

\end{proof}

\section{Conclusion}\label{Conclusion}
In this section we conclude the proof of Theorem \ref{MainThm}. To this aim we give the following definition:
\begin{defi}\label{defin}
Let $u \in C^\infty(\W)$ be given. We define 
\begin{equation}\label{MeanMainC}
v(u)  := \sum_{i=1}^{n+1} \langle u, \, \phi_i\rangle_{L^2} w_i,
\end{equation}
where $\set{w_i}_{i=1}^{n+1}\subset \R^{n+1}$ are chosen such that the associated functions $\phi_i := \phi_{w_i}$ (defined as in \eqref{PhiDef}) are an orthonormal frame in $L^2$ for the vector space $\set{\phi_c}_{c \in \erre^{n+1}}$. We will denote $\phi_u:= \phi_{v(u)}$. 
\end{defi}
We will need the following proposition, whose proof is postponed at the end of this section:
\begin{prop}\label{MeanMainEstimate}
There exists $C=C(n, \, p, \, F)>0$ such that, for every $u \in C^\infty(\W)$, it holds
\begin{equation}\label{MeanMainEstimateEq}
\norm{u - \phi_u}_{W^{2, \, p}(\W)} \le C \inf_{c \in \erre^{n+1}} \norm{u - \phi_c}_{W^{2, \, p}(\W)}.
\end{equation}
\end{prop}

\begin{proof}[Proof of Theorem \ref{MainThm}]
Let $\epsilon>0$ to be fixed small enough at the end of the argument. Let $\delta_0 \leq \delta$, where $\delta$ is as in Theorem \ref{QualitativeAppThm}. By Theorem \ref{QualitativeAppThm} combined with Proposition \ref{C1Close}, $\Sigma=\partial U$ is a radially parametrized hypersurface as in \eqref{RadialPar1} so that the radius $u$ satisfies inequalities \eqref{AppC0} and \eqref{AppC1}. We can consequently apply Theorem \ref{MainEstimateb} to get that $\Sigma$ satisfies \eqref{MainEstimateEq}.

We notice that, for every $c \in U$, we can define 
\begin{equation}
\psi_c \daA{\W}{\Sigma - c },\quad \psi_c(x) := x + u_c(x) \nu(x).
\end{equation}
For every $c$ the mapping $\psi_c$ is an alternative radial parametrization for $\Sigma$, and it is a well defined diffeomorphism. We also define:
\begin{equation*}
\Phi \daA{U}{\erre^{n+1}},\quad \Phi(c):= \sum_{i=1}^{n+1} \langle u_c, \, \phi_i\rangle_{L^2} w_i=v(u_c),
\end{equation*}
where $\set{w_i}_{i=1}^{n+1}$ are as in Definition \ref{defin}.
Our idea is to prove the existence of $c_0 \in U$ such that $\Phi(c_0)=0$. This is enough to conclude the proof, because $\Phi(c_0)=0$ implies that $\phi_{u_{c_0}}=\langle \Phi(c_0), \nu \rangle =0$, which, together with Proposition \ref{MeanMainEstimate}, imply
\begin{equation}\label{eua}
 \norm{u_{c_0}}_{W^{2, \, p}_\sigma(\W)} \le C \coup*{ \norm{\Sdot_F}_{L^p_g(\Sigma)} + \sqrt{\epsilon} \, \norm{u_{c_0}}_{W^{2, \, p}_\sigma(\W)} }  .
\end{equation}
Therefore, if we set $\epsilon_0 = \left(\frac{1}{2C}\right)^2$, then the second term in the right hand side of \eqref{eua} can be absorbed in the left hand side, obtaining  
\begin{equation}\label{eua1}
\norm{\psi_{c_0} - \id}_{W^{2, \, p}_\sigma(\W)}= \norm{u_{c_0}}_{W^{2, \, p}_\sigma(\W)}\leq C  \norm{\Sdot_F}_{L^p_g(\Sigma)}.
\end{equation}
Defining $\psi:=\psi_{c_0}+c_0$, the estimate \eqref{eua1} reads exactly as \eqref{MainEstimate}, which is the desired conclusion of Theorem \ref{MainThm}.

We are just left to find $c_0\in U$ such that $\Phi(c_0)=0$. First of all, it is easy to notice that there exist $\tilde{\epsilon}$ and $\tilde{r}$ depending only on $n$ and $\W$, such that, for every $ 0<\epsilon < \tilde{\epsilon}$, if $\Sigma= \partial U$ satisfies
\[
d_{\hd}\coup*{ U, \, U_{\W} } < \epsilon, 
\] 
then the ball $B^{n+1}_{\tilde{r}}$ is contained in $U$. Hence we consider $0< \epsilon < \tilde{\epsilon}$ so 
small that the ball $B^{n+1}_{\tilde{a} \, \epsilon}$ is contained in $U$ for some $\tilde{a}$, depending only on $n$ and $\W$, that we will choose later. We study $\Phi$ inside $B^{n+1}_{\tilde{a} \, \epsilon}$.  We will show that $\Phi$ admits the following linearisation:
\begin{equation}\label{PhiLinearised}
\Phi(c)= \Phi(0) - c + O_{n, \, \W}(\epsilon^{\frac{3}{2} })\ \mbox{ for every } c \in B^{n+1}_{\tilde{a} \, \epsilon}.
\end{equation}
Indeed, for every $c$ such that $\abs{c} < \tilde{a} \, \epsilon$ we find 
\[
d_{\hd}\coup*{\Sigma -c, \, \W} \le d_{\hd}\coup*{\Sigma -c, \, \Sigma} + d_{\hd}\coup*{\Sigma, \, \W} \le (\tilde{a}+1) \epsilon.
\]
Therefore, it is easy to see that also the function $u_c$ satisfies the estimates 
\begin{align}
\norm{u_c}_{C^0} &\le C \epsilon, \label{1}\\
\norm{\nabla u_c}_{C^0} &\le C \sqrt{\epsilon}, \label{2}
\end{align}
for some $C$ depending only on $n$ and $\W$. We start the linearisation with the following simple consideration: for every $z \in \W$ there exists $x_c=x_c(z)\in\W$ so that
\[
\psi_c(z) = \psi(x_c) - c.
\]
We expand this equality and find
\begin{equation}\label{Equalityzxc}
 z + u_c (z) \, \nu(z) = x_c + u(x_c) \nu(x_c) - c .
\end{equation}
Using the $C^1$-smallness of $u$ and $u_c$, we can easily see that $x_c=x_c(z)$ satisfies the relation
\begin{equation}\label{xcApp}
x_c(z) = z + O_{n, \, \W}(\epsilon) .
\end{equation}
This approximation, combined with \eqref{2}, gives an estimate of $u$ close to $z$: 
\begin{equation}\label{Estimate32}
u(x_c(z)) =  u(z) + O_{n, \, \W}\coup*{\epsilon^{\frac{3}{2}}}.
\end{equation}
Using \eqref{gauge}, we evaluate $F^*$ in the point in \eqref{Equalityzxc}:
\begin{align*}
\underbrace{F^*( z + u_c (z) \, \nu(z))}_{= 1 +  u_c(z) \restr{ dF^*}{z}[\nu(z)] + O_{n, \, \W}(\epsilon^2) } = 
\underbrace{F^*( x_c + u (x_c) \, \nu(x_c) -c)}_{= 1 +  u(x_c) \restr{ dF^*}{x_c}[\nu(x_c)] - \restr{ dF^*}{x_c}[c]  + O_{n, \, \W}(\epsilon^2) }.
\end{align*}
Plugging \eqref{Robin} and \eqref{Estimate32} in the previous equality, we obtain
\begin{equation}\label{GoodEstimate}
u_c(z) = u(z) - \underbrace{\langle c, \, \nu(z)\rangle}_{=\phi_c(z)}  + O_{n, \, W} \coup*{\epsilon^{\frac{3}{2}}}.
\end{equation} 
Integrating over $\W$ and using \eqref{GoodEstimate}, we finally obtain \eqref{PhiLinearised}.
We are now ready to show that $0$ is in the range of $\Phi$. We show that there exists $c_0 \in B^{n+1}_{ \tilde{a} \epsilon}$ such that $\Phi(c_0)=0$. Indeed, choosing 
$$\tilde{a}(n, \, \W) = (n+1) \sup_i |w_i| \, \sup_i \norm{\phi_i}_{L^1} +1 $$
 and $\epsilon$ sufficiently small, we have that
\begin{equation*}
\begin{split}
&(n+1) \sup_i |w_i| \, \sup_i \norm{\phi_i}_{L^1} \norm{u}_{C^0} + O_n\coup*{\epsilon^{\frac{3}{2}}} < \tilde{a} \,  \epsilon \\
&\Rightarrow \abs{\Phi(0)  + O_{n, \, \W}\coup*{\epsilon^{\frac{3}{2}}}} < \tilde{a} \, \epsilon \\
&\Rightarrow \exists c_0\in B^{n+1}_{ \tilde{a} \epsilon}\, :\, c_0 = \Phi(0)  + O_{n, \, \W}\coup*{\epsilon^{\frac{3}{2}}} \\
&\overset{\eqref{PhiLinearised}}{\Rightarrow}\exists c_0\in B^{n+1}_{ \tilde{a} \epsilon}\, :\, \Phi(c_0) = 0,
\end{split}
\end{equation*}
as desired.
\end{proof}

\appendix \label{APPA}
\section{Proof of computational propositions}
In this section we prove the computational propositions stated in Section \ref{Quan}. 

\begin{proof}[Proof of Proposition \eqref{GeomApp}]
Let $x \in \Sigma$, and let $\set{z_1, \, \dots \, z_n}$ be a frame for $T_x \W$. We compute the differential $d\psi$ in these coordinates, obtaining 
\begin{equation}\label{PsiDiff}
\nabla_i \psi = z_i + \nabla_i u \, \nu + u \nabla_i  \nu .
\end{equation}
Using \ref{PsiDiff}, we compute the expression for $g$.
\begin{align*}
g_{ij} 
&= \langle \nabla_i \psi, \, \nabla_j \psi \rangle = \langle  z_i + \nabla_i u \, \nu + u \nabla_i  \nu, \,  z_j + \nabla_j u \, \nu + u \nabla_j  \nu  \rangle \\
&= \omega_{ij} + 2 u h_{ij} + \nabla_i u \, \nabla_j u + u^2 \underbrace{\langle \nabla_i \nu, \, \nabla_j \nu \rangle}_{= h^k_i h_{kj}}
\end{align*}
and \eqref{gApp} follows immediately. Estimates \eqref{gInvApp} and \eqref{detApp} are easy consequences of \eqref{gApp}.

Now we prove \eqref{nuApp}. We exhibit the exact expression for $\nu_\Sigma$.  We firstly search for a vector $ V = \nu + a^i z_i $ which satisfies  the condition $\langle V , \nabla_j \psi \rangle = 0$ for every $j=1,\dots,n$:
\begin{align*}
0 
&= \langle V, \, \nabla_j \psi \rangle =  \langle  \nu + a^i z_i , \, z_j + \nabla_j u \, \nu + u \nabla_j  \nu \rangle  \\
&=  \nabla_j u + (\omega + u h)^i_j a_i.
\end{align*}
Since $\nu_\Sigma= \frac{V}{|V|}$, we obtain the expression for $\nu_\Sigma$:
\begin{equation}\label{nuPrecise}
\nu_\Sigma = \frac{\nu - \coup*{\omega + u h}^{-1}[\nabla u]}{\abs{\nu - \coup*{\omega + u h}^{-1}[\nabla u]}}
\end{equation}
and we easily find \eqref{nuApp}. Estimate \eqref{hApp} follows by the very definition of $h_\Sigma$. Indeed, we write
\[
\nu_{\Sigma} = \nu - \nabla u + \mathcal{R}
\]
where $\mathcal{R}$ is a combination of products of $u$ and $\nabla u$. We obtain  
\begin{align*}
(h_{\Sigma})_{ij} = \langle \nabla_i \psi, \, \nabla_j \nu_\Sigma \rangle = \langle z_i + \nabla_i u \, \nu + u \nabla_i \nu, \, \nabla_j \coup*{\nu - \nabla u + \mathcal{R}}  \rangle
\end{align*}
and \eqref{hApp} easily follows.

To prove \eqref{secondaforma}, since $\restr{A^F}{\nu_\W}h_\W=S_F(\W)=\id$, we compute:
\begin{equation*}
\begin{split}
&\left|S_F(\Sigma)-S_F(\W)+\nabla(\restr{A^F}{\nu_\W}\nabla u)-h_\W u\right|\\
&=\left|\restr{A^F}{\nu_\Sigma}h_\Sigma-\restr{A^F}{\nu_\W}h_\W+\nabla(\restr{A^F}{\nu_\W}\nabla u)-\restr{A^F}{\nu_\W}h_\W^2 u \right|.
\end{split}
\end{equation*}
Plugging in the previous equation \eqref{hApp} and \eqref{nuApp}, we conclude:
\begin{equation*}
\begin{split}
&\left|S_F(\Sigma)-S_F(\W)+\nabla(\restr{A^F}{\nu_\W}\nabla u)-h_\W u\right|\\
&\overset{\eqref{hApp}}{\leq}\left|\restr{A^F}{\nu_\Sigma}(h_\W- \nabla^2 u + h_\W^2 u)-\restr{A^F}{\nu_\W}h_\W\right.\\
&\quad \left. +\nabla(\restr{A^F}{\nu_\W}\nabla u)-\restr{A^F}{\nu_\W}h_\W^2 u \right|+  C \sqrt{\epsilon} \coup*{\abs{u} + \abs{\nabla u} + \abs{\nabla^2 u }}\\
&\overset{\eqref{nuApp}}{\leq}\left|-\restr{A^F}{\nu_\W} \nabla^2 u + \restr{A^F}{\nu_\W} h_\W^2 u-\nabla \restr{A^F}{\nu_\W}\nabla u h_\W\right.\\
&\quad \left. +\nabla(\restr{A^F}{\nu_\W}\nabla u)-\restr{A^F}{\nu_\W}h_\W^2 u \right|+  C \sqrt{\epsilon} \coup*{\abs{u} + \abs{\nabla u} + \abs{\nabla^2 u }}\\
&\leq  C \sqrt{\epsilon} \coup*{\abs{u} + \abs{\nabla u} + \abs{\nabla^2 u }},
\end{split}
\end{equation*}
which gives \eqref{secondaforma}.

In the same way, one can obtain also \eqref{curvatura}.
\end{proof}

\begin{proof}[Proof of Proposition \eqref{C1Close}]
The smallness of $u$ stated in \eqref{AppC0} is trivial, so let us focus on the smallness of $\nabla u$. We claim that if $\Sigma$ satisfies the hypothesis of Proposition \ref{C1Close}, then its tangent planes are near to the tangent planes of $\W$, namely 
\begin{equation}\label{PlaneApp}
d(T_x \W, \, T_{\psi(x)} \Sigma) \le C \sqrt{\epsilon},
\end{equation}
where $C=C(\W)$ and $d$ denotes the distance between elements of the $n$-dimensional Grassmannian (see \cite{SimonLN}). Then \eqref{AppC1} follows because $T_{\psi(x)} \Sigma = \restr{d \psi}{x}[T_x \W]$ and $d\psi$ can be written as 
\[
d\psi = \id + \nabla u \otimes \nu + u \, d\nu = \id + \nabla u \otimes \nu + O(\epsilon)
\]
concluding the proof. 

We are just left to prove \eqref{PlaneApp}. Let $U_\W$, $U_\Sigma$ be the open, convex and bounded sets eclosed by $\W$ and $\Sigma$. Since $\Sigma$ satisfies \eqref{AppC0}, up to center the two hypersurfaces, we can find two radii $r$ and $R$ depending only on $\W$ such that 
\[
B^{n+1}_r \subset U_\W \subset B^{n+1}_R, \qquad B^{n+1}_r \subset U_\Sigma \subset B^{n+1}_R.
\] 
We prove \eqref{PlaneApp} pointwise. Let $x\in \W$, and assume $T_x \W = \erre^n \times \set{0}$. We can write $\W$ and $\Sigma$ in local graph charts near $x$, and $\psi(x)$, and find two maps $\phi_\Sigma$, $\phi_\W$ such that 
\begin{align*}
&\phi_\W \daA{B^n_r}{\W}, \quad \phi_\W(z)= (z, \, f(z)), \, \phi_\W(0) = x \\
&\phi_\Sigma \daA{B^n_r}{\Sigma}, \quad \phi_\Sigma(z)= (z, \, g(z)), \, \phi_\Sigma(0) = \psi(x).
\end{align*}
The graph functions $f$ and $g$ are smooth, convex functions satisfying 
\begin{align*}
\norm{f - g}_0 \le \epsilon \mbox{ in } B^n_r \\
r \le f \le R, \qquad r \le g \le R.
\end{align*}
Therefore, by convexity (see \cite{Convex}), we obtain an upper bound on their Lipschitz constant:
\[
\Lip\coup*{f, \, B^n_{\frac{r}{2}}}, \, \Lip\coup*{g, \, B^n_{\frac{r}{2}}} \le \frac{4(R-r)}{r}=: L = L(\W)
\] 
We want to prove that $f$ and $g$ are $C^1$ close: this will conclude the proof, because $T_x \W = \restr{d \phi_\W}{0}[\erre^n]$, $T_{\psi(x)} \Sigma = \restr{d \phi_\Sigma}{0}[\erre^n]$ and therefore a $C^1$ closeness between the maps easily implies a $C^0$ closeness between the tangent planes. This $C^1$ closeness comes from Lemma \ref{Semiconvexity}, stated and proved right after this proof.

Lemma \ref{Semiconvexity} assures that $g-f$ satisfies the estimate 
\[
\abs{D(g - f)} \le 2 \sqrt{c(f) \epsilon},
\]
where $c(f)= \norm{D^2 f}_{C^0(B^n_r)}$. We show that $c(f)$ can be bounded by a constant independently of $f$. Indeed, we notice that in graph coordinates the mean curvature of $\W$ can be written as 
\[
H = \divv\coup*{\frac{Df}{\sqrt{1 + \abs{D f}^2}}}.
\]
Therefore, by elliptic regularity, we are able to find a constant $c=c(n, \, \alpha, \, r, \, R)>0$ such that 
\[
\norm{f}_{C^{2, \, \alpha}(B^n_{\frac{r}{2}})} \le c \coup*{  \norm{H}_{C^{0, \, \alpha}(B^n_{\frac{r}{2}})} + \norm{f}_{C^{0, \, \alpha}( \partial B^n_{\frac{r}{2}})} } \le c_0 =C_0(\W)
\]
and the thesis follows. The estimate on the perimeter now follows easily by estimate \eqref{detApp}.
\end{proof}

We conclude Appendix A with this lemma, which has been used in the previous proof:
\begin{lemma}\label{Semiconvexity}
Let $h\in C^\infty (\erre^n)$ be a $c_0$ semiconvex function in the ball $B_r^n$, namely 
\[
h + \frac{c_0}{2} \abs{\cdot}^2 \mbox{ is convex in }B_r^n.
\]
Assume $\abs{h} \le \epsilon$ in $B_r^n$. Then $\abs{D h} \le 2\sqrt{c_0 \, \epsilon}$ in $B_{\frac{r}{2}}^n$. 
\end{lemma} 
\begin{proof}
Let $x_0 \in \overline{B^n_r}$ and $v \in \esse^{n-1}$ be given such that $\langle Dh(x_0), \, v \rangle = \norm{Dh}_{C^0 (B_{\frac{r}{2}})}$. We write $h$ as:
\[
h(x_0 + \tau v) - h(x_0) = \tau \underbrace{ \langle Dh(x_0), \, v \rangle}_{= \norm{Dh}_0} + \frac{\tau^2}{2} \int_0^1 t \int_0^1 \restr{ D^2 h }{x_0 + s t \tau v}[v, \, v] .
\]
Now, by the very definition of semiconvexity, we know that $D^2 h \ge - c_0 \Id$ at every point, and plugging this inequality into the previous expression, we obtain the estimate
\[
\norm{D h}_0 \le \frac{2\epsilon}{\tau} + \frac{\tau}{2} c_0.
\]
Choosing $\tau = \frac{2\sqrt{\epsilon}}{\sqrt{c_0}}$, we obtain the thesis.
\end{proof}

\section{A characterization of the stability operator}
The proof of Theorem \ref{MainEstimateb} in Section \ref{Quan} would have been simpler if we had a precise characterization of the kernel of the stability operator. Unfortunately, we were not able to completely characterize it. Still, we can prove a partial characterization of it under an additional condition.
\begin{prop}\label{CharactL}
Denote by $L$ the stability operator defined as
\begin{equation}\label{LOperator}
L[u] := \divv \coup*{A^F \nabla u} + H u.
\end{equation}
The following equality holds: 
\begin{equation}\label{KerL}
\Set{u \in \ker (L) : \int_\W u = 0} = \Set{\phi_c \daA{y \in \W}{\langle c ,\nu(y) \rangle \in \erre}, \, \forall c \in \erre^{n+1}}.
\end{equation}
\end{prop}
\begin{proof}
Let $u\in \ker (L)$ with null mean, and let $v \in C^\infty(\W)$ be a positive function with $\int_\W v = 1$. Using a construction shown in \cite{doCarmo} , we are able to find a function $s \daA{[0, \, \epsilon)}{\erre}$ such that the deformation
\[
\psi_t \daA{\W}{\erre^{n+1}},\ \psi_t(x) = x + (t u(x) + s(t) v(x)) \nu(x) 
\]
is volume preserving for every $t$. Moreover, the construction satisfies also $s(0)=\dot{s}(0)=0$. For every $c \in \erre^{n+1}$, we define the translation
\[
\psi^c_t \daA{\W}{\erre^{n+1}},\ \psi^c_t(x) = x + tc := x + t\phi_c(x) \, \nu(x) + t\xi_c(x),
\]
where we have set $\xi_c := c - \phi_c \, \nu$. 

By the very definition of the $L^1$ norm, we find
\begin{equation}\label{L1Diff}
\norm{\psi^c_t - \psi_t}_{L^1} = \abs{ U_{\W_t} \, \Delta \, U_{\W + tc} },
\end{equation}
where $U_{\W_t}$ and $U_{\W + tc}$ denote the open sets enclosed respectively by $\W_t:= \psi_t(\W)$ and $\W + tc$. 
Moreover, we can also write 
\begin{equation}\label{L1Dec}
\norm{\psi^c_t - \psi_t}_{L^1}= \norm{t(\phi_c- u) \nu - s(t) v \nu + t\xi_{c} }_{L^1}
\end{equation}
From \eqref{L1Dec} and \eqref{L1Diff} we obtain the inequality
\begin{equation}\label{L1Ineq}
t \norm{(\phi_c- u) \nu + \xi_c}_{L^1} \le  s(t) + \abs{ U_{\W_t} \, \Delta \, U_{\W + tc} } .
\end{equation}
However, since $\nu$ and $\xi_c$ are pointwise orthogonal, we find
\[
\norm{(\phi_c- u) \nu + \xi_c}_{L^1} \ge \norm{u - \phi_c}_{L^1}.
\]
Hence we obtain 
\begin{equation}\label{L1IneqDue}
t \norm{u - \phi_c}_{L^1} \le  s(t) + \abs{ U_{\W_t} \, \Delta \, U_{\W + tc} } .
\end{equation}
We take the infimum in $c$ in \eqref{L1IneqDue} and obtain the expression
\begin{equation}
t \inf_{c \in \erre^{n+1}} \norm{u - \phi_c}_{L^1} \le   s(t) + \inf_{c \in \erre^{n+1}} \abs{ U_{\W_t} \, \Delta \, U_{\W + c} } .
\end{equation}
Since $U_{\W_t}$ and $U_{\W + c} $ share the same volume, we can apply Theorem \ref{DeficiThm} to deduce that
\begin{equation}\label{L1End}
t \inf_{c \in \erre^{n+1}} \norm{u - \phi_c}_{L^1} \le   s(t) +  C(n) \sqrt{\PF(\W_t) - \PF(\W)}.
\end{equation}
The computations made in \cite[Prop. 2.1]{KoisoPalmer} gives us the following equalities:
\[
\int_{\W} u = 0 \Rightarrow \restr{ \frac{d}{dt} \PF(\W_t) }{t=0} = 0,\quad L[u]=0 \Rightarrow \restr{ \frac{d^2}{dt^2} \PF(\W_t) }{t=0} = 0,
\]
which, plugged in \eqref{L1End}, give
\begin{equation}
t \inf_{c \in \erre^{n+1}} \norm{u - \phi_c}_{L^1} \le   s(t) + O\coup*{ t^{\frac{3}{2}} }.
\end{equation}
Dividing by $t$ and letting $t \to 0$, we obtain
\[
\inf_{c \in \erre^{n+1}} \norm{u - \phi_c}_{L^1} = 0
\]
and since the infimum is attained, the thesis is proven.
\end{proof}

\end{document}